\documentclass{amsart}
\usepackage{amsmath}
\usepackage{amssymb}
\usepackage{amsfonts}

\newtheorem{theorem}{Theorem}[section]
\theoremstyle{plain}

\newtheorem{corollary}{Corollary}[section]

\newtheorem{remark}{Remark}

\numberwithin{equation}{section}

\begin{document}

 \title[New bounds for the companion of Ostrowski's inequality]{New bounds for the companion of Ostrowski's inequality and applications}
\author[W. J. Liu]{Wenjun Liu}
\address[W. J. Liu]{College of Mathematics and Statistics\\
Nanjing University of Information Science and Technology \\
Nanjing 210044, China}
\email{wjliu@nuist.edu.cn}



\begin{abstract}
In this paper we establish some new bounds for the companion of Ostrowski's inequality for the case when $f'\in L^1[a,b]$, $f''\in L^2[a,b]$ and  $f'\in L^2[a,b]$, respectively. We point out that the results in the first and third cases are sharp and that some of these new estimations can be better than the known results. Some applications to composite quadrature rules, and to probability density functions are also given.
\end{abstract}

\maketitle

\section{Introduction}
An extensive literature deals with inequalities between an integral $\frac{1}{b-a}\int_{a}^{b}f(t)dt$ and its various approximations, such as trapezoidal approximation, midpoint approximation, Simpson approximation and so on.
In 1938, Ostrowski established the following interesting integral
inequality (see \cite{o1938}) for differentiable mappings with bounded
derivatives which generalizes the
estimate of an integral by the midpoint rule:
\begin{theorem}\label{Th1.1}
Let $f:[a,b]\rightarrow\mathbb{R}$ be a differentiable mapping on
$(a,b)$ whose derivative is bounded on $(a,b)$ and denote
$\|f'\|_{\infty}=\displaystyle{\sup_{t\in(a,b)}}|f'(t)|<\infty$.
Then for all $x\in[a,b]$ we have
\begin{equation}\label{1.1}
\left|f(x)-\frac{1}{b-a}\int_{a}^{b}f(t)dt\right|\leq\left[\frac{1}{4}
+\frac{(x-\frac{a+b}{2})^{2}}{(b-a)^{2}}\right](b-a)\|f'\|_{\infty}.
\end{equation}
The constant $\frac{1}{4}$ is sharp in the sense that it can not be
replaced by a smaller one.
\end{theorem}

In \cite{gs2002}, Guessab and Schmeisser proved  the following companion of
Ostrowski's inequality:
\begin{theorem}\label{Th1.2}
Let $f:[a,b]\rightarrow\mathbb{R}$ be satisfying the Lipschitz condition, i.e., $|f(t)-f(s)|\le M |t-s|$.
Then for all $x\in[a,\frac{a+b}{2}]$ we have
\begin{equation}\label{1.2}
\left|\frac{f(x)+f(a+b-x)}{2}-\frac{1}{b-a}\int_{a}^{b}f(t)dt\right|\leq\left[\frac{1}{8}+2\left(\frac{x-\frac{3a+b}{4}}{b-a}\right)^{2}\right]
(b-a) M.
\end{equation}
The constant $\frac{1}{8}$ is sharp in the sense that it can not be
replaced by a smaller one.  In \eqref{1.2}, the  point
 $x=\frac{3a+b}{4}$ gives the best estimator and yields the
trapezoid type inequality, i.e.,
\begin{align}
\left|\frac{f\left(\frac{3a+b}{4}\right)+f\left(\frac{a+3b}{4}\right)}{2}
-\frac{1}{b-a}\int_{a}^{b}f(t)dt\right|\leq\frac{b-a }{8}M.\label{1.3}
\end{align}
The constant $\frac{1}{8}$ in \eqref{1.3} is sharp in the sense mentioned above.
\end{theorem}

Motivated by \cite{gs2002}, Dragomir \cite{d2005} proved some companions of Ostrowski's
inequality, as follows:
\begin{theorem}\label{Th1.3}
Let $f:[a,b]\rightarrow\mathbb{R}$ be an absolutely continuous
mapping on $[a,b]$. Then the following inequalities
\begin{align*}
&\left|\frac{f(x)+f(a+b-x)}{2}-\frac{1}{b-a}\int_{a}^{b}f(t)dt\right|
\nonumber
\\
\leq&\left\{
{\begin{array}{l}\left[\frac{1}{8}+2\left(\frac{x-\frac{3a+b}{4}}{b-a}\right)^{2}\right]
(b-a)\|f'\|_{\infty},\quad f'\in L^{\infty}[a,b], \\
\frac{2^{1/q}}{(q+1)^{1/q}}\left[\left(\frac{x-a}{b-a}\right)^{q+1}+
\left(\frac{\frac{a+b}{2}-x}{b-a}\right)^{q+1}\right]^{1/q}(b-a)^{1/q}\|f'\|_{p},\\
\quad \quad \quad\quad\quad
p>1, \frac{1}{p}+\frac{1}{q}=1\quad \text{and}\quad f'\in L^{p}[a,b], \\
\left[\frac{1}{4}+\left|\frac{x-\frac{3a+b}{4}}{b-a}\right|\right]\|f'\|_{1},
\quad\quad\quad \quad\quad\quad f'\in L^{1}[a,b] \\
\end{array}}\right.
\end{align*}
hold for all $x\in[a,\frac{a+b}{2}]$.
\end{theorem}

Recently, Alomari \cite{a20111} studied the companion of Ostrowski inequality \eqref{1.2} for differentiable bounded mappings.

\begin{theorem}\label{Th1.4}
Let $f:[a,b]\rightarrow\mathbb{R}$ be a differentiable mapping in $(a,b)$. If $f'\in L^1[a,b]$ and $\gamma\le f'(x)\le \Gamma, \forall\ x\in [a,b]$, then for all $x\in[a,\frac{a+b}{2}]$ we have
\begin{equation}
\left|\frac{f(x)+f(a+b-x)}{2}-\frac{1}{b-a}\int_{a}^{b}f(t)dt\right|\leq\left[\frac{1}{16}+ \left(\frac{x-\frac{3a+b}{4}}{b-a}\right)^{2}\right]
(b-a) (\Gamma-\gamma).\label{1.5'}
\end{equation}
\end{theorem}

For other related results, the reader may be refer to \cite{a20112, a20113, a1995, a2008, bdg2009, d20051,  d2011, d2001, hn2011, l2008, l2010, lxw2010, l2007, l2009, s20101, s2010, ss2011, thd2008, thyc2011, u2003, v2011} and the
references therein.

The main aim of this paper is to establish new estimations of the left part of \eqref{1.5'} for the case when $f'\in L^1[a,b]$, $f''\in L^2[a,b]$ and  $f'\in L^2[a,b]$, respectively. It turns out that some of these new estimations can be
better than the known results and that the results in the first and third cases are sharp.
Some applications to composite quadrature rules, and to probability density functions are also given.

\section{Main results}
\subsection{The case when $f'\in L^1[a,b]$ and is bounded}

\begin{theorem}\label{Th2.1}
Let $f:[a,b]\rightarrow\mathbb{R}$ be a differentiable mapping in $(a,b)$. If $f'\in L^1[a,b]$ and $\gamma\le f'(x)\le \Gamma, \forall\ x\in [a,b]$, then for all $x\in[a,\frac{a+b}{2}]$ we have
\begin{equation}
\left|\frac{f(x)+f(a+b-x)}{2}-\frac{1}{b-a}\int_{a}^{b}f(t)dt\right|\leq
\left[\frac{b-a}{4}+\left|x-\frac{3a+b}{4}\right|\right]  (S-\gamma)\label{2.1}
\end{equation}
and
\begin{equation}
\left|\frac{f(x)+f(a+b-x)}{2}-\frac{1}{b-a}\int_{a}^{b}f(t)dt\right|\leq
\left[\frac{b-a}{4}+\left|x-\frac{3a+b}{4}\right|\right]  (\Gamma-S),\label{2.2}
\end{equation}
where $S=(f(b)-f(a))/(b-a).$  If $\gamma, \Gamma$ are
given by
$$\gamma=\inf\limits_{t\in [a,b]}f'(t),\ \
\Gamma=\sup\limits_{t\in [a,b]}f'(t), $$
 then  the constant $\frac{1}{4}$ in \eqref{2.1}  and \eqref{2.2} is sharp in the sense that it can not be
replaced by a smaller one.
\end{theorem}
\begin{proof}
Define the kernel $K(x,t)$ by
\begin{align}
K(x,t):=\left\{ {\begin{array}{l} t-a, \quad \quad t\in[a,x], \\
t-\frac{a+b}{2}, \quad t\in(x,a+b-x], \\
t-b, \quad \quad t\in(a+b-x,b],\\
\end{array}}\right. \label{2.3}
\end{align}
for all $x\in[a,\frac{a+b}{2}]$. Integrating by parts, we obtain
(see  \cite{d2005})
\begin{align}
\frac{1}{b-a}\int_{a}^{b}K(x,t)f'(t)dt=\frac{f(x)+f(a+b-x)}{2}-\frac{1}{b-a}\int_{a}^{b}f(t)dt.
\label{2.4}
\end{align}
We also have
\begin{align}
\frac{1}{b-a}\int_{a}^{b}K(x,t) dt=0
\label{2.5}
\end{align}
and
\begin{align}
\int_{a}^{b}f'(t) dt=f(b)-f(a).
\label{2.6}
\end{align}
From \eqref{2.4}-\eqref{2.6}, it follows that
\begin{align}
&\frac{1}{b-a}\int_{a}^{b}K(x,t)f'(t)dt-\frac{1}{(b-a)^2}\int_{a}^{b}K(x,t) dt\int_{a}^{b}f'(t) dt\nonumber\\
=&\frac{f(x)+f(a+b-x)}{2}-\frac{1}{b-a}\int_{a}^{b}f(t)dt.
\label{2.7}
\end{align}

We denote
\begin{align}R_n(x)=\frac{1}{b-a}\int_{a}^{b}K(x,t)f'(t)dt-\frac{1}{(b-a)^2}\int_{a}^{b}K(x,t) dt\int_{a}^{b}f'(t) dt.
\label{2.8'}
\end{align}
If $C\in \mathbb{R}$ is an arbitrary constant, then we have
\begin{align}
R_n(x)=\frac{1}{b-a}\int_{a}^{b}(f'(t)-C)\left[K(x,t)-\frac{1}{b-a}\int_{a}^{b}K(x,s) ds\right]dt,
\label{2.8}
\end{align}
since $$\int_{a}^{b} \left[K(x,t)-\frac{1}{b-a}\int_{a}^{b}K(x,s) ds\right]dt=0.$$
Furthermore, we have
\begin{align}
|R_n(x)|\le \frac{1}{b-a} \max\limits_{t\in [a,b]}\left|K(x,t)-0\right|\int_{a}^{b}|f'(t)-C|dt.
\label{2.9}
\end{align}
and \begin{align}\max\limits_{t\in [a,b]}\left|K(x,t)\right|=\max\left\{x-a, \frac{a+b}{2}-x\right\}=\frac{b-a}{4}+\left|x-\frac{3a+b}{4}\right|,\quad x\in\left[a,\frac{a+b}{2}\right].\label{2.10}
\end{align}
We also have (see \cite{u2003})
\begin{align}\int_{a}^{b}|f'(t)-\gamma|dt=(S-\gamma)(b-a)\label{2.11}
\end{align}
and
\begin{align}\int_{a}^{b}|f'(t)-\Gamma|dt=(\Gamma-S)(b-a).\label{2.12}
\end{align}
Therefore, we obtain \eqref{2.1} and \eqref{2.2}  by using \eqref{2.7}-\eqref{2.12} and choosing $C = \gamma$  and $C = \Gamma$ in \eqref{2.9}, respectively.

For the sharpness of the constant $\frac{1}{4}$ in \eqref{2.1}, assume that  \eqref{2.1} holds with a constant $A>\frac{1}{4}$, i.e.,
\begin{equation}
\left|\frac{f(x)+f(a+b-x)}{2}-\frac{1}{b-a}\int_{a}^{b}f(t)dt\right|\leq
\left[A(b-a)+\left|x-\frac{3a+b}{4}\right|\right]  (S-\gamma).\label{2.1a}
\end{equation}
For simplicity we take $a=0, b=1, x\in \left[\frac{1}{4}, \frac{1}{2}\right]$ and choose $0<\varepsilon\ll 1$ be small. If we choose $f:[a,b]\rightarrow\mathbb{R}$ with
\begin{equation*}
f(t)=\left\{ {\begin{aligned}  & -\frac{\varepsilon^2}{2},   &&t\in[0, x-\varepsilon-\varepsilon^2], \\
& \frac{[t-(x-\varepsilon-\varepsilon^2)]^2}{2\varepsilon^2}-\frac{\varepsilon^2}{2},   &&t\in[x-\varepsilon-\varepsilon^2, x-\varepsilon],  \\
& t-(x-\varepsilon),  &&t\in(x-\varepsilon, x-\varepsilon^2],\\
& -\frac{(t-x)^2}{2\varepsilon^2}+\varepsilon-\frac{\varepsilon^2}{2},   &&t\in(x-\varepsilon^2, x],  \\
& \varepsilon-\frac{\varepsilon^2}{2},   &&t\in(x, 1],
\end{aligned}}\right.
\end{equation*}
then
\begin{equation*}
f'(t)=\left\{ {\begin{aligned}  & 0,   &&t\in[0, x-\varepsilon-\varepsilon^2], \\
& \frac{t-(x-\varepsilon-\varepsilon^2)}{\varepsilon^2},   &&t\in[x-\varepsilon-\varepsilon^2, x-\varepsilon],  \\
& 1,  &&t\in(x-\varepsilon, x-\varepsilon^2],\\
& -\frac{ t-x }{ \varepsilon^2},   &&t\in(x-\varepsilon^2, x],  \\
& 0,   &&t\in(x, 1],
\end{aligned}}\right.
\end{equation*}
which implies that $f$ is differentiable in $(a,b)$ and $\gamma=\inf\limits_{t\in [a,b]}f'(t)=0$, and
$$S=\frac{f(b)-f(a)}{b-a} =\varepsilon,\quad \frac{f(x)+f(a+b-x)}{2}=\varepsilon-\frac{\varepsilon^2}{2},\quad \frac{1}{b-a}\int_{a}^{b}f(t)dt=\frac{1}{2}\varepsilon(2+\varepsilon^2-2x),$$
giving in \eqref{2.1} $\varepsilon x-\frac{\varepsilon^2+\varepsilon^3}{2}\le \varepsilon(A+x-\frac{1}{4})$. Therefore we get $A\ge \frac{1}{4}$ when  $0<\varepsilon\ll 1$ is chosen to be small enough.

The sharpness of the constant $\frac{1}{4}$ in \eqref{2.2} can be proved similarly.
\end{proof}

\begin{corollary}
Under the assumptions of Theorem \ref{Th2.1} with $x=\frac{3a+b}{4}$, we have the
trapezoid type inequalities
\begin{align}
\left|\frac{f\left(\frac{3a+b}{4}\right)+f\left(\frac{a+3b}{4}\right)}{2}
-\frac{1}{b-a}\int_{a}^{b}f(t)dt\right|\leq \frac{b-a}{4}(S-\gamma),\label{2.13}
\end{align}
\begin{align}
\left|\frac{f\left(\frac{3a+b}{4}\right)+f\left(\frac{a+3b}{4}\right)}{2}
-\frac{1}{b-a}\int_{a}^{b}f(t)dt\right|\leq \frac{b-a}{4}(\Gamma-S).\label{2.14}
\end{align}
\end{corollary}

\begin{corollary}
Under the assumptions of Theorem \ref{Th2.1} with $x=a$, we have the
trapezoid inequalities
\begin{align}
\left|\frac{f\left(a\right)+f\left(b\right)}{2}
-\frac{1}{b-a}\int_{a}^{b}f(t)dt\right|\leq \frac{b-a}{2}(S-\gamma),\label{2.15}
\end{align}
\begin{align}
\left|\frac{f\left(a\right)+f\left(b\right)}{2}
-\frac{1}{b-a}\int_{a}^{b}f(t)dt\right|\leq \frac{b-a}{2}(\Gamma-S).\label{2.16}
\end{align}
\end{corollary}

\begin{corollary}
Under the assumptions of Theorem \ref{Th2.1} with $x=\frac{a+b}{2}$, we have the
midpoint inequalities
\begin{align}
\left|f\left(\frac{a+b}{2}\right)
-\frac{1}{b-a}\int_{a}^{b}f(t)dt\right|\leq \frac{b-a}{2}(S-\gamma),\label{2.17}
\end{align}
\begin{align}
\left|f\left(\frac{a+b}{2}\right)
-\frac{1}{b-a}\int_{a}^{b}f(t)dt\right|\leq \frac{b-a}{2}(\Gamma-S).\label{2.18}
\end{align}
\end{corollary}

\begin{remark} \label{re0}
We note that \eqref{2.15}-\eqref{2.16}, and \eqref{2.17}-\eqref{2.18} can also be obtained by choosing $x=a$ and $x=\frac{a+b}{2}$ in \cite[Theorem 3]{u2003}, respectively.
In fact, in \cite[Theorem 3]{u2003},  the inequalities
 \begin{align*}
\left|f(x)-\left(x-\frac{a+b}{2}\right)\frac{f(b)-f(a)}{b-a}
-\frac{1}{b-a}\int_{a}^{b}f(t)dt\right|\leq \frac{b-a}{2}(S-\gamma),\quad x\in [a,b]
\end{align*}
 \begin{align*}
\left|f(x)-\left(x-\frac{a+b}{2}\right)\frac{f(b)-f(a)}{b-a}
-\frac{1}{b-a}\int_{a}^{b}f(t)dt\right|\leq \frac{b-a}{2}(\Gamma-S),\quad x\in [a,b]
\end{align*}
 were proved.
 However, it is obvious that  \eqref{2.13} and \eqref{2.14} give  a smaller estimator than the above inequalities.
\end{remark}

A new inequality of Ostrowski's type may be stated as follows:
\begin{corollary}\label{co2.4}
Let $f$ be as in Theorem \ref{Th2.1}. Additionally, if $f$ is symmetric about the line $x=\frac{a+b}{2}$,
i.e., $f (a + b - x) = f (x)$,
 then for all $x\in[a,\frac{a+b}{2}]$ we have
\begin{equation}
\left|f(x)-\frac{1}{b-a}\int_{a}^{b}f(t)dt\right|\leq
\left[\frac{b-a}{4}+\left|x-\frac{3a+b}{4}\right|\right]  (S-\gamma),
\end{equation}\label{2.19}
\begin{equation}
\left|f(x)-\frac{1}{b-a}\int_{a}^{b}f(t)dt\right|\leq
\left[\frac{b-a}{4}+\left|x-\frac{3a+b}{4}\right|\right]  (\Gamma-S).
\end{equation}\label{2.20}
\end{corollary}

\begin{remark} \label{re1}
Under the assumptions of Corollary \ref{co2.4} with $x=a$, we have
\begin{align}
\left|f\left(a\right)
-\frac{1}{b-a}\int_{a}^{b}f(t)dt\right|\leq \frac{b-a}{2}(S-\gamma),\label{2.21}
\end{align}
\begin{align}
\left|f\left(a\right)
-\frac{1}{b-a}\int_{a}^{b}f(t)dt\right|\leq \frac{b-a}{2}(\Gamma-S).\label{2.22}
\end{align}
\end{remark}

\subsection{The case when  $f''\in L^2[a,b]$}

\begin{theorem}\label{Th2.2}
Let $f:[a,b]\rightarrow\mathbb{R}$ be a twice continuously differentiable
mapping  in $(a,b)$ with $f''\in L^2[a,b]$. Then for all $x\in[a,\frac{a+b}{2}]$ we have
\begin{align}
 \left|\frac{f(x)+f(a+b-x)}{2}-\frac{1}{b-a}\int_{a}^{b}f(t)dt\right| 
\leq  
\frac{(b-a)^{1/2}}{\pi}\left[\frac{(b-a)^2}{48}+\left(x-\frac{3a+b}{4}\right)^2\right]^{1/2}\|f''\|_2.\label{2.23}
\end{align}
\end{theorem}
\begin{proof}
Let $R_n(x)$ be defined by \eqref{2.8'}. From \eqref{2.7}, we get
$$R_n(x)=\frac{f(x)+f(a+b-x)}{2}-\frac{1}{b-a}\int_{a}^{b}f(t)dt.$$
If we choose $C = f'((a + b)/2)$ in \eqref{2.8} and use the Cauchy inequality, then we get
\begin{align}
|R_n(x)|\le &\frac{1}{b-a}\int_{a}^{b}\left|f'(t)-f'\left(\frac{a+b}{2}\right)\right|\left|K(x,t)-\frac{1}{b-a}\int_{a}^{b}K(x,s) ds\right|dt
\nonumber\\
\le &\frac{1}{b-a}\left[\int_{a}^{b}\left(f'(t)-f'\left(\frac{a+b}{2}\right)\right)^2dt\right]^{1/2} \left[\int_{a}^{b}\left(K(x,t)-\frac{1}{b-a}\int_{a}^{b}K(x,s) ds\right)^2dt\right]^{1/2}.
\label{2.24}
\end{align}
We can use the Diaz-Metcalf inequality (see \cite[p. 83]{mpf1991} or \cite[p. 424]{u2003}) to get
$$\int_{a}^{b}\left(f'(t)-f'\left(\frac{a+b}{2}\right)\right)^2dt\le \frac{(b-a)^2}{\pi^2}\|f''\|_2^2.$$
We also have
\begin{align}&\int_{a}^{b}\left(K(x,t)-\frac{1}{b-a}\int_{a}^{b}K(x,s) ds\right)^2dt\nonumber\\
=
&\int_{a}^{b} K(x,t)^2dt=\left[\frac{(b-a)^2}{48}+\left(x-\frac{3a+b}{4}\right)^2\right](b-a).\label{2.25'}
\end{align}
Therefore, using the above relations, we obtain \eqref{2.23}.
\end{proof}

\begin{corollary}
Under the assumptions of Theorem \ref{Th2.2} with $x=\frac{3a+b}{4}$, we have the
trapezoid type inequality
\begin{align}
\left|\frac{f\left(\frac{3a+b}{4}\right)+f\left(\frac{a+3b}{4}\right)}{2}
-\frac{1}{b-a}\int_{a}^{b}f(t)dt\right|\leq \frac{(b-a)^{3/2}}{4\sqrt{3}\pi}\|f''\|_2.\label{2.26}
\end{align}
\end{corollary}

\begin{corollary}
Under the assumptions of Theorem \ref{Th2.2} with $x=a$, we have the
trapezoid inequality
\begin{align}
\left|\frac{f\left(a\right)+f\left(b\right)}{2}
-\frac{1}{b-a}\int_{a}^{b}f(t)dt\right|\leq \frac{(b-a)^{3/2}}{2\sqrt{3}\pi}\|f''\|_2.\label{2.27}
\end{align}
\end{corollary}

\begin{corollary}
Under the assumptions of Theorem \ref{Th2.2} with $x=\frac{a+b}{2}$, we have the
midpoint inequalities
\begin{align}
\left|f\left(\frac{a+b}{2}\right)
-\frac{1}{b-a}\int_{a}^{b}f(t)dt\right|\leq \frac{(b-a)^{3/2}}{2\sqrt{3}\pi}\|f''\|_2.\label{2.28}
\end{align}
\end{corollary}

\begin{remark} \label{re2}
We note that \eqref{2.27} and \eqref{2.28} can also be obtained by choosing $x=a$ and $x=\frac{a+b}{2}$ in \cite[Theorem 4]{u2003}, respectively.
In fact, in \cite[Theorem 4]{u2003},  the inequality
 \begin{align*}
\left|f(x)-\left(x-\frac{a+b}{2}\right)\frac{f(b)-f(a)}{b-a}
-\frac{1}{b-a}\int_{a}^{b}f(t)dt\right|\leq \frac{(b-a)^{3/2}}{2\sqrt{3}\pi}\|f''\|_2,\quad x\in [a,b]
\end{align*}
 was proved.
 However, it is obvious that  \eqref{2.26} gives a smaller estimator than the above inequality.
\end{remark}

The other new inequality of Ostrowski's type may be stated as follows:
\begin{corollary}\label{co2.8}
Let $f$ be as in Theorem \ref{Th2.2}. Additionally, if $f$ is symmetric about the line $x=\frac{a+b}{2}$,
i.e., $f (a + b - x) = f (x)$,
 then for all $x\in[a,\frac{a+b}{2}]$ we have
\begin{equation}
\left|f(x)-\frac{1}{b-a}\int_{a}^{b}f(t)dt\right|\leq
\frac{(b-a)^{1/2}}{\pi}\left[\frac{(b-a)^2}{48}+\left(x-\frac{3a+b}{4}\right)^2\right]^{1/2}\|f''\|_2.\label{2.29}
\end{equation}
\end{corollary}

\begin{remark} \label{re3}
Under the assumptions of Corollary \ref{co2.8} with $x=a$, we have
\begin{align}
\left|f\left(a\right)
-\frac{1}{b-a}\int_{a}^{b}f(t)dt\right|  \leq \frac{(b-a)^{3/2}}{2\sqrt{3}\pi}\|f''\|_2.\label{2.30}
\end{align}
\end{remark}

\subsection{The case when $f'\in L^2[a,b]$}

\begin{theorem}\label{Th2.3}
Let $f:[a,b]\rightarrow\mathbb{R}$ be an absolutely continuous mapping in $(a,b)$ with $f'\in L^2[a,b]$. Then for all $x\in[a,\frac{a+b}{2}]$ we have
\begin{align}
&\left|\frac{f(x)+f(a+b-x)}{2}-\frac{1}{b-a}\int_{a}^{b}f(t)dt\right|\nonumber \\
\leq &
(b-a)^{-1/2} \left[\frac{(b-a)^2}{48}+\left(x-\frac{3a+b}{4}\right)^2\right]^{1/2}\sqrt{\sigma(f')},\label{2.31}
\end{align}
where $\sigma(f')$ is defined by
$$\sigma(f')=\|f'\|_2^2-\frac{(f(b)-f(a))^2}{b-a}=\|f'\|_2^2-S^2(b-a)$$
and $S$ is defined in Theorem \ref{Th2.1}. Inequality \eqref{2.31} is sharp in the sense that the constant $\frac{1}{48}$ of the right-hand side cannot be replaced by a
smaller one.
\end{theorem}
\begin{proof}
Let $R_n(x)$ be defined by \eqref{2.8'}. From \eqref{2.7}, we get
$$R_n(x)=\frac{f(x)+f(a+b-x)}{2}-\frac{1}{b-a}\int_{a}^{b}f(t)dt.$$
If we choose $C = \frac{1}{b-a}\int_{a}^{b}f'(s)ds$ in \eqref{2.8} and use the Cauchy inequality and \eqref{2.25'}, then we get
\begin{align}
&|R_n(x)|\nonumber\\
\le &\frac{1}{b-a}\int_{a}^{b}\left|f'(t)-\frac{1}{b-a}\int_{a}^{b}f'(s)ds\right|\left|K(x,t)-\frac{1}{b-a}\int_{a}^{b}K(x,s) ds\right|dt
\nonumber\\
\le &\frac{1}{b-a}\left[\int_{a}^{b}\left(f'(t)-\frac{1}{b-a}\int_{a}^{b}f'(s)ds \right)^2dt\right]^{1/2} \left[\int_{a}^{b}\left(K(x,t)-\frac{1}{b-a}\int_{a}^{b}K(x,s) ds\right)^2dt\right]^{1/2}\nonumber \\
\leq &\sqrt{\sigma(f')}
\left[\frac{(b-a)^2}{48}+\left(x-\frac{3a+b}{4}\right)^2\right]^{1/2}(b-a)^{-1/2}.
\label{2.33}
\end{align}

The sharpness of the constant $\frac{1}{48}$ in \eqref{2.31} can be obtained in a particular case for $x=a$ or $x=\frac{a+b}{2}$ which has been
proved in \cite[Theorem 3 and Remark 2]{l2007} and  \cite[Propositions 2.3 and 2.6]{d2001}.
\end{proof}

\begin{corollary}
Under the assumptions of Theorem \ref{Th2.2} with $x=\frac{3a+b}{4}$, we have the
trapezoid type inequality
\begin{align}
\left|\frac{f\left(\frac{3a+b}{4}\right)+f\left(\frac{a+3b}{4}\right)}{2}
-\frac{1}{b-a}\int_{a}^{b}f(t)dt\right|\leq \frac{(b-a)^{1/2}}{4\sqrt{3} }\sqrt{\sigma(f')}.\label{2.34}
\end{align}
\end{corollary}

\begin{corollary}
Under the assumptions of Theorem \ref{Th2.2} with $x=a$, we have the sharp
trapezoid inequality
\begin{align}
\left|\frac{f\left(a\right)+f\left(b\right)}{2}
-\frac{1}{b-a}\int_{a}^{b}f(t)dt\right|\leq \frac{(b-a)^{1/2}}{2\sqrt{3}}\sqrt{\sigma(f')}.\label{2.35}
\end{align}
\end{corollary}

\begin{corollary}
Under the assumptions of Theorem \ref{Th2.2} with $x=\frac{a+b}{2}$, we have the sharp
midpoint inequalities
\begin{align}
\left|f\left(\frac{a+b}{2}\right)
-\frac{1}{b-a}\int_{a}^{b}f(t)dt\right|\leq \frac{(b-a)^{1/2}}{2\sqrt{3}}\sqrt{\sigma(f')}.\label{2.36}
\end{align}
\end{corollary}

\begin{remark} \label{re4}
We note that \eqref{2.35} and \eqref{2.36} are also given  in \cite[Propositions 2.3 and 2.6]{d2001} and \cite[Remark 2]{l2007}. However, it is obvious that  \eqref{2.34} gives a smaller estimator, and can neither be obtained from \cite[Theorem 3 and Remark 2]{u2003} for any special case, nor from  \cite[Propositions 2.3 and 2.6]{d2001}.
\end{remark}

Another inequality of Ostrowski's type may be stated as follows:
\begin{corollary}\label{co2.12}
Let $f$ be as in Theorem \ref{Th2.3}. Additionally, if $f$ is symmetric about the line $x=\frac{a+b}{2}$,
i.e., $f (a + b - x) = f (x)$,
 then for all $x\in[a,\frac{a+b}{2}]$ we have
\begin{equation}
\left|f(x)-\frac{1}{b-a}\int_{a}^{b}f(t)dt\right|\leq
 (b-a)^{-1/2} \left[\frac{(b-a)^2}{48}+\left(x-\frac{3a+b}{4}\right)^2\right]^{1/2}\sqrt{\sigma(f')}.\label{2.37}
\end{equation}
\end{corollary}

\begin{remark} \label{re5}
Under the assumptions of Corollary \ref{co2.12} with $x=a$, we have
\begin{align}
\left|f\left(a\right)
-\frac{1}{b-a}\int_{a}^{b}f(t)dt\right|  \leq \frac{(b-a)^{1/2}}{2\sqrt{3} }\sqrt{\sigma(f')}.\label{2.38}
\end{align}
\end{remark}

\section{Application to Composite Quadrature Rules}

Let $I_{n}: a=x_{0}<x_{1}<\cdot\cdot\cdot<x_{n-1}<x_{n}=b$ be a
partition of the interval $[a,b]$ and $h_{i}=x_{i+1}-x_{i}$
$(i=0,1,2,\cdot\cdot\cdot,n-1)$.
Consider the general quadrature formula
\begin{equation}
S(f,I_{n})=\frac{1}{2}\sum_{i=0}^{n-1}\left[f\left(\frac{3x_{i}+x_{i+1}}{4}\right)+f\left(\frac{x_{i}+3x_{i+1}}{4}\right)\right]h_{i}. \label{3.1}
\end{equation}

\begin{theorem}\label{Th3.1}
Let $f:[a,b]\rightarrow\mathbb{R}$ be a differentiable mapping in $(a,b)$. If $f'\in L^1[a,b]$ and $\gamma\le f'(x)\le \Gamma, \forall\ x\in [a,b]$, then we have
\begin{equation*}
\int_{a}^{b}f(x)dx=S(f,I_{n})+R(f,I_{n})
\end{equation*}
and the remainder $R(f,I_{n})$ satisfies the estimates
\begin{equation}
|R(f,I_{n})|\leq\frac{1}{4}\sum_{i=0}^{n-1}(S_i-\gamma)h_i^2 \label{3.2}
\end{equation}
and
\begin{equation}
|R(f,I_{n})|\leq\frac{1}{4}\sum_{i=0}^{n-1}(\Gamma-S_i)h_i^2, \label{3.3}
\end{equation}
where $S_i=(f(x_{i+1})-f(x_{i}))/h_i, i=0,1,2,\cdot\cdot\cdot,n-1.$
\end{theorem}
\begin{proof}
Applying \eqref{2.13} to the interval $[x_{i},x_{i+1}]$,
then we get
\begin{equation*}
\left|\frac{1}{2}\left[f\left(\frac{3x_{i}+x_{i+1}}{4}\right)+f\left(\frac{x_{i}+3x_{i+1}}{4}\right)\right]h_{i}
- \int_{x_{i}}^{x_{i+1}}f(t)dt\right|\leq \frac{h_i^2}{4}(S_i-\gamma)
\end{equation*}
for $i=0,1,2,\cdot\cdot\cdot,n-1.$
Now summing over $i$ from $0$ to $n-1$  and using the triangle
inequality,  we get \eqref{3.2}. In a similar way, we get  \eqref{3.3}.
\end{proof}

\begin{remark} \label{re7}
It is obvious that the estimations obtained in Theorem \ref{Th3.1} are better than those of  \cite[Theorem 7]{u2003} due to a smaller error.
\end{remark}

\begin{theorem}\label{Th3.2}
Let $h_{i}=x_{i+1}-x_{i}=h=\frac{b-a}{n}$
$(i=0,1,2,\cdot\cdot\cdot,n-1)$ and let  $f:[a,b]\rightarrow\mathbb{R}$ be a twice continuously differentiable
mapping  in $(a,b)$ with $f''\in L^2[a,b]$. Then we have
\begin{equation*}
\int_{a}^{b}f(x)dx=S(f,I_{n})+R(f,I_{n})
\end{equation*}
and the remainder $R(f,I_{n})$ satisfies the estimate
\begin{equation}
|R(f,I_{n})|\leq \frac{(b-a)^{5/2}}{4\sqrt{3}\pi n^2}\|f''\|_2.\label{3.4}
\end{equation}
\end{theorem}
\begin{proof}
Applying \eqref{2.26} to the interval $[x_{i},x_{i+1}]$,
then we get
\begin{equation*}
\left|\frac{h}{2}\left[f\left(\frac{3x_{i}+x_{i+1}}{4}\right)+f\left(\frac{x_{i}+3x_{i+1}}{4}\right)\right]
- \int_{x_{i}}^{x_{i+1}}f(t)dt\right|\leq \frac{h^{5/2}}{4\sqrt{3}\pi}\left[\int_{x_i}^{x_{i+1}}(f''(t))^2 dt\right]^{1/2}
\end{equation*}
for $i=0,1,2,\cdot\cdot\cdot,n-1.$
Now summing over $i$ from $0$ to $n-1$,  and using the triangle
inequality and the Cauchy inequality,  we get
\begin{align*}
&\left|\frac{h}{2}\sum_{i=0}^{n-1}\left[f\left(\frac{3x_{i}+x_{i+1}}{4}\right)+f\left(\frac{x_{i}+3x_{i+1}}{4}\right)\right]
- \int_{a}^{b}f(t)dt\right|\\
\leq &\frac{h^{5/2}}{4\sqrt{3}\pi}\sum_{i=0}^{n-1}\left[\int_{x_i}^{x_{i+1}}(f''(t))^2 dt\right]^{1/2}
\leq  \frac{h^{5/2}}{4\sqrt{3}\pi}\sqrt{n}\left[\sum_{i=0}^{n-1}\int_{x_i}^{x_{i+1}}(f''(t))^2 dt\right]^{1/2}
=  \frac{(b-a)^{5/2}}{4\sqrt{3}\pi n^2}\|f''\|_2.
\end{align*}
Therefore,  \eqref{3.4} is obtained.
\end{proof}

\begin{theorem}\label{Th3.3}
Let $h_{i}=x_{i+1}-x_{i}=h=\frac{b-a}{n}$
$(i=0,1,2,\cdot\cdot\cdot,n-1)$ and let $f:[a,b]\rightarrow\mathbb{R}$ be an absolutely continuous mapping in $(a,b)$ with $f'\in L^2[a,b]$.  Then we have
\begin{align*}
\int_{a}^{b}f(x)dx=S(f,I_{n})+R(f,I_{n})
\end{align*}
and the remainder $R(f,I_{n})$ satisfies the estimate
\begin{align}
|R(f,I_{n})|\leq \frac{(b-a)^{3/2}}{4\sqrt{3}n} \sqrt{\sigma(f')}.\label{3.5}
\end{align}
\end{theorem}
\begin{proof}
Applying \eqref{2.34} to the interval $[x_{i},x_{i+1}]$,
then we get
\begin{align*}
&\left|\frac{h}{2}\left[f\left(\frac{3x_{i}+x_{i+1}}{4}\right)+f\left(\frac{x_{i}+3x_{i+1}}{4}\right)\right]
- \int_{x_{i}}^{x_{i+1}}f(t)dt\right|\\
\leq &\frac{h^{3/2}}{4\sqrt{3}}\left[\int_{x_i}^{x_{i+1}}(f'(t))^2 dt-\frac{(f(x_{i+1})-f(x_{i}))^2}{h}\right]^{1/2}
\end{align*}
for $i=0,1,2,\cdot\cdot\cdot,n-1.$
Now summing over $i$ from $0$ to $n-1$, using the triangle
inequality  and using the Cauchy inequality twice,  we get
\begin{align*}
&\left|\frac{h}{2}\sum_{i=0}^{n-1}\left[f\left(\frac{3x_{i}+x_{i+1}}{4}\right)+f\left(\frac{x_{i}+3x_{i+1}}{4}\right)\right]
- \int_{a}^{b}f(t)dt\right|\\
\leq &\frac{h^{3/2}}{4\sqrt{3}}\sum_{i=0}^{n-1}\left[\int_{x_i}^{x_{i+1}}(f'(t))^2 dt-\frac{(f(x_{i+1})-f(x_{i}))^2}{h}\right]^{1/2}\\
\leq & \frac{h^{3/2}}{4\sqrt{3}} \sqrt{n}\left[\|f'\|_2^2-\frac{n}{b-a}\sum_{i=0}^{n-1}(f(x_{i+1})-f(x_{i}))^2\right]^{1/2}\\
\leq & \frac{h^{3/2}}{4\sqrt{3}} \sqrt{n}\left[\|f'\|_2^2-\frac{(f(b)-f(a))^2}{b-a} \right]^{1/2}\\
=& \frac{(b-a)^{3/2}}{4\sqrt{3}n} \sqrt{\sigma(f')}.
\end{align*}
Therefore,  \eqref{3.5} is obtained.
\end{proof}

\section{Application to probability density functions}

Now, let $X$ be a random variable taking values in the finite interval
$[a,b]$, with the probability density function $f : [a, b]\rightarrow [0, 1]$ and with
the cumulative distribution function $$F (x) = Pr (X \leq  x) = \int_a^x f (t) dt.$$

The following results hold:
\begin{theorem}\label{Th4.1}
With the assumptions of Theorem \ref{Th2.1}, we have
\begin{equation}
\left|\frac{1}{2}[F(x)+F(a+b-x)]-\frac{b-E(X)}{b-a}\right|\leq
\left[\frac{b-a}{4}+\left|x-\frac{3a+b}{4}\right|\right]  \left(\frac{1}{b-a}-\gamma\right)\label{4.1}
\end{equation}
and
\begin{equation}
\left|\frac{1}{2}[F(x)+F(a+b-x)]-\frac{b-E(X)}{b-a}\right|\leq
\left[\frac{b-a}{4}+\left|x-\frac{3a+b}{4}\right|\right]  \left(\Gamma-\frac{1}{b-a}\right),\label{4.2}
\end{equation}
for all $x\in[a,\frac{a+b}{2}]$, where $E (X)$ is the expectation of $X$.
\end{theorem}
\begin{proof}  By \eqref{2.1} and \eqref{2.2} on choosing $f = F$ and taking into account
$$E(X)=\int_a^b t dF(t)=b-\int_a^b F(t)dt,$$
we obtain \eqref{4.1} and \eqref{4.2}.
\end{proof}

\begin{corollary}
Under the  assumptions of Theorem \ref{Th4.1} with $x=\frac{3a+b}{4}$, we have
\begin{equation*}
\left|\frac{1}{2}\left[F\left(\frac{3a+b}{4}\right)+F\left(\frac{a+3b}{4}\right)\right]-\frac{b-E(X)}{b-a}\right|\leq
 \frac{b-a}{4}  \left(\frac{1}{b-a}-\gamma\right)
\end{equation*}
and
\begin{equation*}
\left|\frac{1}{2}\left[F\left(\frac{3a+b}{4}\right)+F\left(\frac{a+3b}{4}\right)\right]-\frac{b-E(X)}{b-a}\right|\leq
 \frac{b-a}{4}  \left(\Gamma-\frac{1}{b-a}\right).
\end{equation*}
\end{corollary}

\begin{theorem}\label{Th4.2}
With the assumptions of Theorem \ref{Th2.2}, we have
\begin{equation}
\left|\frac{1}{2}[F(x)+F(a+b-x)]-\frac{b-E(X)}{b-a}\right|\leq
\frac{(b-a)^{1/2}}{\pi}\left[\frac{(b-a)^2}{48}+\left(x-\frac{3a+b}{4}\right)^2\right]^{1/2}\|f'\|_2 \label{4.3}
\end{equation}
for all $x\in[a,\frac{a+b}{2}]$.
\end{theorem}
\begin{proof}  By \eqref{2.23} on choosing $f = F$ and taking into account
$$E(X)=\int_a^b t dF(t)=b-\int_a^b F(t)dt,$$
we obtain \eqref{4.3}.
\end{proof}

\begin{corollary}
Under the  assumptions of Theorem \ref{Th4.2} with $x=\frac{3a+b}{4}$, we have
\begin{equation*}
\left|\frac{1}{2}\left[F\left(\frac{3a+b}{4}\right)+F\left(\frac{a+3b}{4}\right)\right]-\frac{b-E(X)}{b-a}\right|\leq
\frac{(b-a)^{3/2}}{4\sqrt{3}\pi}\|f'\|_2.
\end{equation*}
\end{corollary}

\begin{theorem}\label{Th4.3}
With the assumptions of Theorem \ref{Th2.3}, we have
\begin{equation}
\left|\frac{1}{2}[F(x)+F(a+b-x)]-\frac{b-E(X)}{b-a}\right|\leq
(b-a)^{-1/2} \left[\frac{(b-a)^2}{48}+\left(x-\frac{3a+b}{4}\right)^2\right]^{1/2}\sqrt{\sigma(f)},\label{4.4}
\end{equation}
for all $x\in[a,\frac{a+b}{2}]$, where $\sigma(f) =\|f\|_2^2-\frac{1}{b-a}.$
\end{theorem}
\begin{proof}  By \eqref{2.31} on choosing $f = F$ and taking into account
$$E(X)=\int_a^b t dF(t)=b-\int_a^b F(t)dt,$$
we obtain \eqref{4.4}.
\end{proof}

\begin{corollary}
Under the  assumptions of Theorem \ref{Th4.3} with $x=\frac{3a+b}{4}$, we have
\begin{equation*}
\left|\frac{1}{2}\left[F\left(\frac{3a+b}{4}\right)+F\left(\frac{a+3b}{4}\right)\right]-\frac{b-E(X)}{b-a}\right|\leq
\frac{(b-a)^{1/2}}{4\sqrt{3} }\sqrt{\sigma(f)}.
\end{equation*}
\end{corollary}

\subsection*{Acknowledgments}
This work was partly supported by the National Natural Science Foundation
of China (Grant No. 40975002) and the Natural Science Foundation of the Jiangsu
Higher Education Institutions (Grant No. 09KJB110005). The author would like to thank Professor J. Duoandikoetxea for his constructive comments on earlier versions of this paper.

\end{document}